\documentclass[a4paper,10pt]{article}
\usepackage{latexsym,amsmath,amsthm,amssymb}
\usepackage{a4wide}
\usepackage{hyperref}
\usepackage{marginnote}
\usepackage{color}
\usepackage{cite}
\hypersetup{
pdftitle={CKN--curvature}   
pdfauthor={Van Hoang},
colorlinks = true,
linkcolor = magenta,
citecolor = blue,
}

\theoremstyle{plain}
\newtheorem{theorem}{Theorem}[section]

\newtheorem{proposition}[theorem]{Proposition}

\newtheorem{corollary}[theorem]{Corollary}

\theoremstyle{definition}

\theoremstyle{remark}

\renewcommand{\thefootnote}{\arabic{footnote}}
\newcommand{\C}[1]{\ensuremath{{\mathcal C}^{#1}}} 

\def\R{\mathbb R}
\def\C{\mathbb C}

\def\al{\alpha}
\def\om{\omega}
\def\Om{\Omega}
\def\be{\beta}
\def\ga{\gamma}
\def\de{\delta}
\def\De{\Delta} 

\def\si{\sigma}
\def\lam{\lambda}
\def\vphi{\varphi}
\def\ep{\epsilon}
\def\na{\nabla}
\def\pa{\partial}
\def\la{\langle} 
\def\ra{\rangle} 
\def\lt{\left}
\def\rt{\right}

\def\i0i{\int_0^\infty}

\def\CKNP{(\text{{\bf CKN}})_P}

\def\sf{\text{\rm supp}f}

\def\Vol{\text{Vol}}
\def\bD{\text{\bf D}}

\numberwithin{equation}{section}

\title{Sharp Caffarelli--Kohn--Nirenberg inequalities on Riemannian manifolds: the influence of curvature}
\author{Van Hoang Nguyen
\footnote{Institut de Math\'ematiques de Toulouse, Universit\'e Paul Sabatier, 118 Route de Narbonne, 31062 Toulouse c\'edex 09, France.}
}

\begin{document}
\maketitle


\renewcommand{\thefootnote}{}

\footnote{Email: \href{mailto: Van Hoang Nguyen <van-hoang.nguyen@math.univ-toulouse.fr>}{van-hoang.nguyen@math.univ-toulouse.fr}, \href{mailto: Van Hoang Nguyen <vanhoang0610@yahoo.com>}{vanhoang0610@yahoo.com}}

\footnote{2010 \emph{Mathematics Subject Classification\text}: 26D10, 31C12, 53C20, 53C21.}

\footnote{\emph{Key words and phrases\text}: Caffarelli--Kohn--Nirenberg inequalities, sharp constant, Riemannian manifold, curvature}

\renewcommand{\thefootnote}{\arabic{footnote}}
\setcounter{footnote}{0}

\begin{abstract}
We first establish a family of sharp Caffarelli--Kohn--Nirenberg type inequalities on the Euclidean spaces and then extend them to the setting of Cartan--Hadamard manifolds with the same best constant. The quantitative version of these inequalities also are proved by adding a non-negative remainder term in terms of the sectional curvature of manifolds. We next prove several rigidity results for complete Riemannian manifolds supporting the Caffarelli--Kohn--Nirenberg type inequalities with the same sharp constant as in $\R^n$ (shortly, sharp CKN inequalities). Our results illustrate the influence of curvature to the sharp CKN inequalities on the Riemannian manifolds. They extend recent results of Krist\'aly to a larger class of the sharp CKN inequalities.
\end{abstract}

\section{Introduction}
Let us start by recalling the celebrated interpolation inequalities of order one due to Caffaralli, Kohn and Nirenberg \cite{CKN} (nowaday, called Caffarelli--Kohn--Nirenberg (shortly, CKN) inequalities): let $n\geq 1$ and let $p,q,r, \al, \beta, \gamma, \de$ and $\sigma$ be real number such that
\begin{equation}\label{eq:CKNcond1}
p, q\geq 1,\quad r >0,\quad a \in [0,1],
\end{equation}
and
\begin{equation}\label{eq:CKNcond2}
\frac1p +\frac\al n >0,\quad \frac1q +\frac\beta n>0,\quad \frac1r +\frac\ga n >0,
\end{equation}
where
\begin{equation}\label{eq:CKNcond3}
\gamma = \de \sigma + (1-\de)\beta.
\end{equation}
Then there exists a positive constant $C$ such that the following inequality holds for any function $f\in C^\infty_0(\R^n)$ 
\begin{equation}\label{eq:CKN}
\lt(\int_{\R^n} |f|^r|x|^{r\ga} dx\rt)^{\frac1r} \leq C\lt(\int_{\R^n} |f|^p|x|^{\al p} dx\rt)^{\frac{\de}p} \lt(\int_{\R^n} |f|^q|x|^{\beta q} dx\rt)^{\frac{1-\de}q}
\end{equation}
if and only if the following conditions hold
\begin{equation}\label{eq:CKNbalance}
\frac1r +\frac\ga n = \de\lt(\frac1p +\frac{\al -1}n\rt) + (1-\de)\lt(\frac1q +\frac\be n\rt),
\end{equation}
(this is dimensional balance)
\begin{equation}\label{eq:CKNcond4}
\al -\si \geq 0 \quad\text{ if } \quad \de >0,
\end{equation}
and
\begin{equation}\label{eq:CKNcond5}
\al -\si \leq 1\quad\text{ if } \quad \de >0\quad \text{and}\quad \frac1r + \frac\ga n = \frac1p + \frac{\al -1}n.
\end{equation}
The CKN inequalities contain many well-known inequalities, for examples, Sobolev inequalities, Hardy inequalities, Hardy--Sobolev inequalities, the Gagliardo--Nirenberg inequalities, etc. They play an important role in theory of partial differential equations and have been intensively studied in many settings such as the stratified Lie groups, the homogeneous groups, the metric measure spaces, to Riemannian manifolds with negative curvature and to derivatives of fractional order, etc. We refer the readers to \cite{DD1,DD2,Costa,HAN,HM,Nguyen2017,ORS,RS2017carnot,RS2017L2,RSY2017Lie,RSY2017,RSY2017a,Xia} for mor detailed discussions on this subject.

It is an interesting and non-trivial problem to look for the sharp constant and extremals for the CKN inequalities. Several results are well-known in this direction. For examples, the sharp constants in the Sobolev inequalities was found independently by \cite{Aubin,Talenti}, the sharp Hardy--Sobolev inequalities was proved by Lieb \cite{Lieb}, the sharp constants in the Gagliardo--Nirenberg inequalities was established by Del Pino and Dolbeault \cite{DD1,DD2} (see also \cite{CNV} for different proof by using mass transportation technique), etc. In \cite{Xia}, the sharp constant in a subclass of CKN inequalities was find out: suppose that $r > p> 1$ and $\al, \be, \ga \in \R$ such that
\begin{equation}\label{eq:integrabilitycond}
\frac1r -\frac{\ga}n >0,\quad \frac1p -\frac\al n >0,\quad 1 -\frac\be n >0
\end{equation}
and
\begin{equation}\label{eq:balancecond}
\ga = \frac{1+ \al}{r} + \frac{p-1}{pr} \beta,
\end{equation}
then the following inequality holds for any $f\in C_0^\infty(\R^n)$
\begin{equation}\label{eq:CKNofXia}
\int_{\R^n} \frac{|f|^r}{|x|^{\ga r}} dx \leq \frac{r}{n-\ga r} \lt(\int_{\R^n} \frac{|\na f|^p}{|x|^{\al p}} dx\rt)^{\frac1p} \lt(\int_{\R^n} \frac{|f|^{\frac{p(r-1)}{p-1}}}{|x|^{\beta}} dx\rt)^{\frac{p-1}p}.
\end{equation}
Furthermore, if 
\begin{equation}\label{eq:condofXia}
n-\beta < \lt(1+ \al -\frac\be p\rt)\frac{p(r-1)}{r-p}
\end{equation}
then the inequality \eqref{eq:CKNofXia} is sharp, i.e., the constant $\frac{r}{n-\ga r}$ is the best constant in \eqref{eq:CKNofXia}, and a family of extremal functions is given by
\begin{equation}\label{eq:familyi}
f(x) = c(\lam + |x|^{1+ \al -\frac\beta p})^{\frac{p-1}{p-r}}, \quad c\in \R, \, \lam >0.
\end{equation}
In \cite{Xia}, Xia also proved a rigidity result as follows: let $(M,g)$ be a complete Riemannian manifold with nonnegative Ricci curvature, let $p, q,r, \al, \beta, \ga$ satisfy $r> p> 1$ and the conditions \eqref{eq:integrabilitycond}, \eqref{eq:balancecond} and \eqref{eq:condofXia} and let $P$ be a fixed point in $M$. If the CKN inequality \eqref{eq:CKNofXia} holds on $(M,g)$ with $|x|$ being replaced by $d(P,x)$ the geodesic distance on $M$, then $(M,g)$ is isometric to $\R^n$. This rigidity result can be included into the best constant program initiated by Aubin \cite{Aubin79} and studied by Ledoux \cite{Ledoux}, Bakry, Concordet and Ledoux \cite{Bakry}, Cheeger and Colding \cite{Cheeger}, Druet, Hebey and Vaugon \cite{Druet}, do Carmo and Xia \cite{CarXia}, Minerbe \cite{Minerbe}, Li and Wang \cite{LiWang}, Xia \cite{Xiasobolev, XiaGN,Xia}, Krist\'aly \cite{K1,K2}, Krist\'aly and Ohta \cite{KO}, etc. In the aforementioned papers, the authors show that complete Riemannian manifolds with non-negative Ricci curvature supporting some Sobolev-type inequalities should be close to Euclidean spaces whenever the constant is sufficiently close to the sharp constant in the corresponding inequality in Euclidean space. We refer the reader to the book of Hebey \cite{Hebey} for a thoroughgoing presentation of this subject.

Our origin motivation of this paper is to extend the CKN inequality \eqref{eq:CKNofXia} to a larger class of indices $p,q,r, \al, \be, \ga$. Let $\pa_r = \frac x{|x|} \cdot \na$ denote the radial derivative of functions on $\R^n$. Our first main result of this paper is as follows:

\begin{theorem}\label{myCKN}
Suppose that $n\geq 2$, $p> 1$, $r >0$ and $\al, \beta, \ga$ satisfy the conditions \eqref{eq:integrabilitycond} and \eqref{eq:balancecond}. Then the following inequalities hold for any function $f \in C_0^\infty(\R^n)$:
\item (a) If $r>1$, then we have
\begin{equation}\label{eq:posexponent}
\int_{\R^n} \frac{|f|^r}{|x|^{\ga r}} dx \leq \frac{r}{(n-\ga r)} \lt(\int_{\R^n} \frac{|\pa_r f|^p}{|x|^{\al p}} dx\rt)^{\frac1p} \lt(\int_{\R^n} \frac{|f|^{\frac{p(r-1)}{p-1}}}{|x|^{\beta}} dx\rt)^{\frac{p-1}p}.
\end{equation}
\item (b) If $r\in (0,1)$, then we have
\begin{equation}\label{eq:negexponent}
\int_{\R^n} \frac{|f|^r}{|x|^{\ga r}} dx \leq \frac{r}{(n-\ga r)} \lt(\int_{\R^n} \frac{|\pa_r f|^p}{|x|^{\al p}} dx\rt)^{\frac1p} \lt(\int_{\text{supp}f} \frac{|f|^{\frac{p(r-1)}{p-1}}}{|x|^{\beta}} dx\rt)^{\frac{p-1}p},
\end{equation}
where $\text{supp} f$ denotes the support of function $f$.

Moreover, the constant $\frac{r}{n-\ga r}$ is sharp if one of the following conditions holds:

\item (i) $1 < p < r$ and \eqref{eq:condofXia} holds. A family of extremal is given by
\begin{equation}\label{eq:extremali}
f(x) = c\lt(\lam(\om) + |x|^{1+ \al -\frac\beta p}\rt)^{\frac{p-1}{p-r}}, \quad \om = \frac{x}{|x|},
\end{equation}
where $c$ is constant and $\lam: S^{n-1}\to (0,\infty)$ such that $\int_{S^{n-1}} \lam(\om)^{\frac{p(r-1)}{p-r} + \frac{n-\beta}{1+ \al -\frac\beta p}} d\om < \infty$.
\item (ii) $0 < r < p$, $r \not=1$ and $1+ \al -\frac\beta p >0$. A family of extremal is given by
\begin{equation}\label{eq:extremalii}
f(x) = c \lt(\lam(\om) -|x|^{1+ \al -\frac\beta p}\rt)_+^{\frac{p-1}{p-r}}, \quad \om = \frac{x}{|x|}
\end{equation}
where $a_+ = \max\{a,0\}$ denotes the positive part of a real number $a$, $c$ is a constant, $\lam: S^{n-1} \to (0,\infty)$ such that $\int_{S^{n-1}} \lam(\om)^{\frac{p(r-1)}{p-r} + \frac{n-\beta}{1+ \al -\frac\beta p}} d\om < \infty$.
\item (iii) $0 < r < p$, $r \not=1$ and $1+ \al -\frac\beta p =0$. A family of extremal is given by
\begin{equation}\label{eq:extremaliii}
f(x) = c \lt(\lam(\om) -\ln |x|\rt)_+^{\frac{p-r}{p-1}},\quad \om = \frac x{|x|},
\end{equation}
where $c$ is constant, $\lam : S^{n-1} \to \R$ such that $\int_{S^{n-1}} e^{(n-\beta)\lam} d\om < \infty$
\item (iv) $0 < r < p$, $r \not=1$, $1+ \al -\frac\beta p <0$ and $n-\beta + (1+\al -\frac\beta p) \frac{p(r-1)}{p-r} >0$. A family of extremal is given by
\begin{equation}\label{eq:extremaliv}
f(x) = c\lt(|x|^{1+ \al -\frac\beta p} -\lam(\om)\rt)_+^{\frac{p-1}{p-r}},\quad \om =\frac x{|x|},
\end{equation}
where $c$ is a constant and $\lam: S^{n-1} \to (0,\infty)$ such that $\int_{S^{n-1}} \lam^{\frac{p(r-1)}{p-r} + \frac{n-\beta}{1+ \al -\frac\beta p}} d\om < \infty$.
\item (v) $r =p$ and $1+ \al -\frac\be p>0$. A family of extremal is given by
\begin{equation}\label{eq:extremalv}
f(x) = \lam (\om) e^{-c |x|^{1+ \al -\frac\beta p}},
\end{equation}
where $c>0$ is a constant and $\lam :S^{n-1} \to \R$ such that $\int_{S^{n-1}}|\lam|^{p} d\om < \infty$.

\end{theorem}



Obviously, the inequality $|\pa_r f|\leq |\na f|$ holds. Consequently, the inequality \eqref{eq:posexponent} is stronger than the one of Xia \eqref{eq:CKNofXia}. In general, we have the following consequences.

\begin{corollary}\label{myCKNfull}
Suppose $n, p,q,r,\al, \be,\ga$ as in the statement of Theorem \ref{myCKN}. Then the following inequalities hold true for any $f\in C_0^\infty(\R^n\setminus\{0\})$:
\item (a) If $r>1$, then we have
\begin{equation}\label{eq:posexponentfull}
\int_{\R^n} \frac{|f|^r}{|x|^{\ga r}} dx \leq \frac{r}{(n-\ga r)} \lt(\int_{\R^n} \frac{|\na f|^p}{|x|^{\al p}} dx\rt)^{\frac1p} \lt(\int_{\R^n} \frac{|f|^{\frac{p(r-1)}{p-1}}}{|x|^{\beta}} dx\rt)^{\frac{p-1}p}.
\end{equation}
\item (b) If $r\in (0,1)$, then we have
\begin{equation}\label{eq:negexponentfull}
\int_{\R^n} \frac{|f|^r}{|x|^{\ga r}} dx \leq \frac{r}{(n-\ga r)} \lt(\int_{\R^n} \frac{|\na f|^p}{|x|^{\al p}} dx\rt)^{\frac1p} \lt(\int_{\text{supp}f} \frac{|f|^{\frac{p(r-1)}{p-1}}}{|x|^{\beta}} dx\rt)^{\frac{p-1}p}.
\end{equation}
Moreover, the constant $\frac{r}{n-\ga r}$ is sharp if one the conditions (i)--(v) in Theorem \ref{myCKN} holds true, and all extremal are given by the corresponding family of extremal in Theorem \ref{myCKN} with $\lam$ being identically constant.
\end{corollary}


The novelty in Theorem \ref{myCKN} and Corollary \ref{myCKNfull} is that the inequalities are established for radial derivation $\pa_r$, a family of extremal is found out, and especially it extends the inequality of Xia \eqref{eq:CKNofXia} to the case $0< r <p$. It is remarkable that if $1 < r < (2p-1)/p$ or $r \in (0,1)$ then $q: =p(r-1)/(p-1) \in (0,1)$ or $q < 0$ respectively. Hence, Theorem \ref{myCKN} and Corollary \ref{myCKNfull} also provide the new type of CKN inequalities in comparing with the one of Caffarelli, Kohn and Nirenberg \eqref{eq:CKN}. Also by $q < 0$ if $r\in (0,1)$, we then need a slight modification in inequalities \eqref{eq:negexponent} and \eqref{eq:negexponentfull} by taking the second integral in their right hand side on $\sf$ to ensure these inequalities being sense. Note that the celebrated sharp Heisenberg--Pauli--Weyl uncertainty principle (see \cite{Heisenberg,Weyl}) and its extremals (up to a constant) by the family of Gaussian functions is a special case Corollary \ref{myCKNfull} above corresponding to the case $p=r =2$, $\al =0$ and $\beta =-2$.

The next purpose of this paper is to describe a complete scenario concerning to the CKN inequalities \eqref{eq:posexponent} and \eqref{eq:negexponent} on complete Riemannian manifolds. Our next results tell us that the inequalities \eqref{eq:posexponent} and \eqref{eq:negexponent} still hold on Cartan--Hadamard manifolds $(M,g)$ (i.e., $n-$dimensional complete simply connected Riemannian manifolds with non-positive sectional curvature). For a Riemannian manifold $(M,g)$ with Riemannian metric $g$, we denote by $\na_g f$ the gradient of function $f$ with respect to metric $g$ and by $|\na_g f| = \sqrt{\la \na_g f,\na_g f\ra}$ the length of $\na_g f$, and by $d_P(x) = d(x,P), x\in M$ for a fixed point $P\in M$, where $d$ is geodesic distance on $M$. We also use $\pa_\rho f$ to denote the radial derivation of function $f$ on $M$ (the derivation along the geodesic curve starting from a fixed point $P\in M$). We then have the following results.

\begin{theorem}\label{CKNonCH}
Let $n, p, q, r, \al, \beta$ and $\ga$ be as in statement of Theorem \ref{myCKN}. Let $(M,g)$ be an $n-$dimensional Cartan--Hadamard manifold and $P \in M$ be a fixed point, and $\pa_\rho$ denote the radial derivation along geodesic curve starting from $P$. Then the following inequalities hold true for any function $f\in C_0^\infty(M)$:
\item (a) If $r>1$, then we have
\begin{equation}\label{eq:posexponentCH}
\int_{\R^n} \frac{|f|^r}{d_P(x)^{\ga r}} dx \leq \frac{r}{(n-\ga r)} \lt(\int_{\R^n} \frac{|\pa_\rho f|^p}{d_P(x)^{\al p}} dx\rt)^{\frac1p} \lt(\int_{\R^n} \frac{|f|^{\frac{p(r-1)}{p-1}}}{d_P(x)^{\beta}} dx\rt)^{\frac{p-1}p}.
\end{equation}
\item (b) If $r\in (0,1)$, then we have
\begin{equation}\label{eq:negexponentCH}
\int_{\R^n} \frac{|f|^r}{d_P(x)^{\ga r}} dx \leq \frac{r}{(n-\ga r)} \lt(\int_{\R^n} \frac{|\pa_\rho f|^p}{d_P(x)^{\al p}} dx\rt)^{\frac1p} \lt(\int_{\text{supp}f} \frac{|f|^{\frac{p(r-1)}{p-1}}}{d_P(x)^{\beta}} dx\rt)^{\frac{p-1}p}.
\end{equation}
Moreover, the constant $\frac{r}{n-\ga r}$ is sharp if one of the conditions (i)--(v) in Theorem \ref{myCKN} holds.
\end{theorem}
Theorem \ref{CKNonCH} to gethether with Gauss lemma which says that $|\pa_\rho f| \leq |\na_\rho f|$ implies the following extension of Corollary \ref{myCKNfull} to the Cartan--Hadamard manifolds.

\begin{corollary}\label{CKNonCHfull}
Suppose the assumptions of Theorem \ref{CKNonCH}. Then the following inequalities hold true for any function $f\in C_0^\infty(M)$:
\item (a) If $r>1$, then we have
\begin{equation}\label{eq:posexponentCHfull}
\int_{\R^n} \frac{|f|^r}{d_P(x)^{\ga r}} dx \leq \frac{r}{(n-\ga r)} \lt(\int_{\R^n} \frac{|\na_g f|^p}{d_P(x)^{\al p}} dx\rt)^{\frac1p} \lt(\int_{\R^n} \frac{|f|^{\frac{p(r-1)}{p-1}}}{d_P(x)^{\beta}} dx\rt)^{\frac{p-1}p}.
\end{equation}
\item (b) If $r\in (0,1)$, then we have
\begin{equation}\label{eq:negexponentCHfull}
\int_{\R^n} \frac{|f|^r}{d_P(x)^{\ga r}} dx \leq \frac{r}{(n-\ga r)} \lt(\int_{\R^n} \frac{|\na_g f|^p}{d_P(x)^{\al p}} dx\rt)^{\frac1p} \lt(\int_{\text{supp}f} \frac{|f|^{\frac{p(r-1)}{p-1}}}{d_P(x)^{\beta}} dx\rt)^{\frac{p-1}p}.
\end{equation}
Moreover, the constant $\frac{r}{n-\ga r}$ is sharp if one of the conditions (i)--(v) in Theorem \ref{myCKN} holds.
\end{corollary}
It is worthy to note that if the sectional curvature of $(M,g)$ is bounded from above by a strict negative constant then the CKN inequalities in Theorem \ref{CKNonCH} and Corollary \ref{CKNonCHfull} can be strengthened by adding a non-negative remainder term concerning to the upper bound of sectional curvature (see Section 3 below for more details).

In the sequel, we characterize the complete Riemannian manifolds which support the sharp CKN inequalities in Corollary \ref{myCKNfull} (i.e., inequalities \eqref{eq:posexponentfull} and \eqref{eq:negexponentfull}. Hereafter, in order to avoid confusions, the sharpness is understood in the sense that the CKN inequalities (of type \eqref{eq:posexponentfull} or \eqref{eq:negexponentfull}) hold on a Riemannian manifold $(M,g)$ with the same best constant as in the Euclidean space. From now on, we always assumptions that $n\geq 2$, $p>1$, $r>0$ and $\al, \be, \ga$ satisfy the conditions \eqref{eq:integrabilitycond} and \eqref{eq:balancecond}. Note that both \eqref{eq:posexponent} and \eqref{eq:negexponent} can be written in the form
\begin{equation*}\label{eq:formchung}
\int_{\R^n} \frac{|f|^r}{|x|^{\ga r}} dx \leq \frac{r}{(n-\ga r)} \lt(\int_{\R^n} \frac{|\na f|^p}{|x|^{\al p}} dx\rt)^{\frac1p} \lt(\int_{\text{\rm supp}f} \frac{|f|^{\frac{p(r-1)}{p-1}}}{|x|^{\beta}} dx\rt)^{\frac{p-1}p}.
\end{equation*}
Let $(M,g)$ be a $n-$dimensional  complete Riemannian manifold, $dV_g$ denote its canonical volume element, and $d_P(x) = d(x,P)$ be the geodesic distance from a point $x\in M$ to a point $P\in M$. For $P\in M$ fixed, we consider the CKN inequalities on $(M,g)$ at $P$ (shortly, $\CKNP$) of the form: for all $f \in C_0^\infty(M)$
\begin{equation*}
\int_{M} \frac{|f|^r}{d_P(x)^{\ga r}} dV_g \leq \frac{r}{(n-\ga r)} \lt(\int_{M} \frac{|\na_g f|^p}{d_P(x)^{\al p}} dV_g\rt)^{\frac1p} \lt(\int_{\sf} \frac{|f|^{\frac{p(r-1)}{p-1}}}{d_P(x)^{\beta}} dV_g\rt)^{\frac{p-1}p}. \tag{$\CKNP$}
\end{equation*}
Corollary \ref{CKNonCHfull} says that $\CKNP$ holds true on $n-$dimensional Cartan--Hadamard manifolds $(M,g)$. Our next result characterizes the attainability of the sharp constant in $\CKNP$ as follows.
\begin{theorem}\label{Maintheorem1}
Given $n\geq 2$, $r\geq p >1$ and $\al, \be, \ga \in \R$ satisfying the conditions \eqref{eq:integrabilitycond} and \eqref{eq:balancecond}. Suppose an extra assumption that $n-\beta + (1+ \al -\frac{\beta}p) \frac{p(r-1)}{p-r}< 0$ if $1< p< r$, or $1+ \al -\frac\beta p >0$ if $r=p$. Let $(M,g)$ be an $n-$dimensional Cartan--Hadamard manifold. Then the following statements are equivalent:
\begin{description}
\item (a) $\frac{r}{n-\ga r}$ is achived by an extremal which is not identically zero in $\CKNP$ for some $P \in M$.
\item (b) $\frac{r}{n-\ga r}$ is achived by an extremal which is not identically zero in $\CKNP$ for all $P \in M$.
\item (c) $M$ is isometric to $\R^n$.
\end{description}
\end{theorem}
The case $p=2$ in Theorem \ref{Maintheorem1} was proved by Krist\'aly (see, e.g., \cite[Theorems $1.1$ and $1.3$]{K2}) under an assumption that the extremal is positive in $M$. This extra assumption is removed in our theorem. Theorem \ref{Maintheorem1} gives a non-positively curved counterpart of the rigidity result of Xia (see \cite[Theorem $1.3$]{Xia}) which asserts that if a complete Riemannian manifolds $(M,g)$ with non-negative Ricci curvature supporting the CKN inequality \eqref{eq:posexponent} with $r>p>1$ and $\al, \beta, \ga$ satisfying the conditions  \eqref{eq:integrabilitycond}, \eqref{eq:balancecond} and $n-\beta + (1+ \al -\frac{\beta}p) \frac{p(r-1)}{p-r}< 0$ must be isometric to $\R^n$. We refer readers to \cite{Ledoux,Bakry,Druet,CarXia,Xiasobolev,Xia,XiaGN,K1,K2,KO,LiWang,Minerbe,Hebey,Cheeger}) for another results in this subject. We next prove such a rigid result in the case $p=r$. It contains a recent rigidity result of Krist\'aly \cite[Theorem $1.2$]{K2} for the sharp Heisenberg--Pauli--Weyl principle (i.e., the case $p=r=2$, $\al =0$ and $\beta =-2$ of \eqref{eq:posexponentfull}) as a special case.

\begin{theorem}\label{Maintheorem2}
Given $n\geq 2$, $p>1$ and $\al, \beta, \ga \in \R$ such that \eqref{eq:integrabilitycond} and \eqref{eq:balancecond} hold true with $r =p$. Suppose, in addition, that $1 + \al -\beta/p>0$. Let $(M,g)$ be a $n-$dimensional complete Riemannian manifold with non-negative Ricci curvature. Then the following statements are equivalent:
\begin{description}
\item (a) $\CKNP$ holds for some $P \in M$.
\item (b) $\CKNP$ holds for all $P \in M$.
\item (c) $M$ is isometric to $\R^n$.
\end{description}
\end{theorem}

We next consider the case $0< r < p$, $r \not =1$. As seen before (see Theorem \ref{myCKN} and Corollary \ref{myCKNfull}), the extremal of \eqref{eq:posexponent} and \eqref{eq:negexponent} in the Euclidean space with $0< r < p$, $r\not=1$ are compactly supported functions. This is very different with the case $r\geq p>1$ in which the extremal never vanish. Consequently, different with the result in Theorem \ref{Maintheorem1} in which a global result was proved, the attainability of the sharp constants in $\CKNP$ only characterizes locally the Riemannian manifold $(M,g)$ around the point $P\in M$ as stated in the following theorem.

\begin{theorem}\label{Maintheorem3}
Given $n\geq 2$, $p>1$, $0 < r < p, r\not=1$ and $\al, \beta,\ga \in \R$ satisfying the conditions \eqref{eq:integrabilitycond} and \eqref{eq:balancecond}. Suppose that one of the following extra assumptions holds: $1+ \al -\frac\beta p \geq 0$ or $1+ \al -\frac\beta p < 0$ and $n-\beta + (1+ \al -\frac{\beta}p) \frac{p(r-1)}{p-r}>0$. Let $(M,g)$ be an $n-$dimensional Cartan--Hadamard manifold, and $P\in M$. Then the following statements are equivalent:
\begin{description}
\item (a) $\frac{r}{n-\ga r}$ is achived by an extremal which is not identically zero in $\CKNP$.
\item (b) There exists $r_P >0$ such that the geodesic ball $B(P,r_P)$ is isometric to $B_{r_P}(0)$, here $B_r(0)$ denotes the ball in $\R^n$ with center at origin and radius $r$.
\end{description}
\end{theorem}

We next present a non-negative curved counterpart of Theorem \ref{Maintheorem3}, that is, an analogue of Theorem \ref{Maintheorem2} and Theorem $1.3$ of Xia \cite{Xia} in the case $r \geq p >1$. We will see that in the non-negatively curve case, the situation is even more rigid than in Theorem \ref{Maintheorem3}. 

\begin{theorem}\label{Maintheorem4}
Given $n\geq 2$, $p>1$, $0< r< p$, $r\not=1$ and $\al, \beta, \ga \in \R$ satisfying the conditions \eqref{eq:balancecond} and \eqref{eq:integrabilitycond}. Suppose that one of the following extra assumptions holds: $1 + \al -\frac\beta p \geq 0$ or $1+ \al -\frac\beta p < 0$ and $n-\beta + (1+ \al -\frac{\beta}p) \frac{p(r-1)}{p-r}>0$. Let $(M,g)$ be an $n-$dimensional complete Riemannian manifold with non-negative Ricci curvature. Then the following statements are equivalent:
\begin{description}
\item (a) $\CKNP$ holds for some $P \in M$.
\item (b) $\CKNP$ holds for all $P \in M$.
\item (c) $M$ is isometric to $\R^n$.
\end{description}
\end{theorem}

The rest of this paper is organized as follows. In Section 2, we recall the notion and results from Riemannian geometry which are used throughout in our proofs. In Section 3, we first prove the sharp CKN inequalities in Theorem \ref{myCKN} in Euclidean space, and then extend them to Cartan--Hadamard manifolds (i.e., prove Theorem \ref{CKNonCH}). We also prove in this section the quantitative CKN inequalities by adding the nonnegative remainder terms concerning to the upper bound of the sectional curvature of Riemannian manifolds. In Section 4, we prove the rigidity results for Cartan--Hadamard manifolds whenever $\CKNP$ is attained, that is, we prove Theorem \ref{Maintheorem1} and \ref{Maintheorem3}. In Section 5, we prove Theorem \ref{Maintheorem2} and \ref{Maintheorem4} on the rigidity results for complete Riemannian manifolds with non-negative Ricci curvature which support the sharp CKN inequalities.

\section{Preliminaries}
In this section, we list some basic properties on Riemannian manifolds, especially the properties of the Cartan--Hadamard manifolds and complete Riemannian manifolds with non-negative Ricci curvature. Let $(M,g)$ be an $n-$dimensional complete Riemannian manifolds and let $d$ be the geodesic distance associated to the Riemannian metric $g$ on $M$. For each $P\in M$ and $\rho>0$, let $B(P,\rho) = \{x\in M\, :\, d(x,P) < \rho\}$ denote the open geodesic ball with center $P\in M$ and radius $\rho >0$. Let $dV_g$ denote the canonical volume element on $(M,g)$, the volume of a bounded open set $\Om \subset M$ is given by
\[
\Vol_g(\Om) = \int_\Om dV_g.
\]
In general, we have for any $P \in M$ that
\begin{equation}\label{eq:locallimit}
\lim_{\rho\to 0^+} \frac{\Vol_g(B(P,\rho))}{\om_n \rho^n} =1
\end{equation}
where $\om_n$ denotes the volume of unit ball in $\R^n$.

If $\{x^i\}_{i=1}^n$ is a local coordinate system, then we can write
\[
g = \sum_{i,j=1}^n g_{ij} dx^i dx^j.
\]
In such a local coordinate system, the Laplace-Beltrami operator $\De$ with respect to the metric $ds^2$ is of the form
\[
\De_g = \sum_{i,j=1}^n \frac1{\sqrt{g}} \frac{\partial}{\partial x_i}\lt(\sqrt{|g|} g^{ij} \frac{\partial}{\partial x^j}\rt),
\]
where $|g| =\text{\rm det}(g_{ij})$ and $(g^{ij}) = (g_{ij})^{-1}$. Let us denote by $\na_g$ the corresponding gradient. Then
\[
\la \na_g u,\na_g v\ra = \sum_{i,j=1}^n g^{ij} \frac{\pa u}{\pa x^i} \frac{\pa v}{\pa x^j}.
\]
For simplicity, we shall use the notation $|\al|= \sqrt{\la \al, \al\ra}$ for any $1-$form $\al$.

Let $K_M$ be the sectional curvature on $M$. A Riemannian manifold $(M,g)$ is called Cartan--Hadamard manifold if it is complete, simply connected and with nonnegative section curvature (i.e., $K_M \leq 0$ along each plane section at each point of $M$). 

If $(M,g)$ is a Cartan--Hadamard manifolds, then for each point $P\in M$, $M$ contains no points conjugate to $P$, and the exponential map $\text{\rm Exp}_P: T_PM \to M$ is a diffeomorphism, where $T_PM$ is the tangent space to $M$ at $P$ (see, e.g., \cite[Chapter ${\rm I}$]{Helgason}). Fix a point $P\in M$ and denote $M$ by $d_P(x) =d(x,P)$ for all $x \in M$. Note that $d_P(x)$ is smooth on $M\setminus\{P\}$ and satisfies 
\[
|\na_g d_P(x)| =\la \na_g d_P(x),\na_g d_P(x)\ra^{\frac12} =1,\qquad x \in M \setminus\{P\}.
\]
Moreover, since $\text{\rm Exp}_P$ is a diffeomorphism, then the function
\[
d_P(x)^2 = \|\text{\rm Exp}_P^{-1}(x)\|^2 \in C^\infty(M).
\]
The radial derivation $\pa_\rho = \frac{\pa}{\pa \rho}$ is defined for any function $f$ on $M$ by 
\[
\pa_\rho f(x) = \frac{d (f\circ \text{\rm Exp}_P)}{dr}(\text{\rm Exp}_P^{-1}(x)),
\]
where $\frac{d}{dr}$ denotes the radial derivation on $T_PM$, i.e., 
\[
\frac{d}{dr} F(u) =  \frac{\la u, \na F(u)\ra}{|u|},\qquad u\in T_PM \setminus\{0\}.
\]

Let $(M,g)$ be a complete Riemannian manifold. We introduce the density function $J(u,t)$ of the volume form in normal coordinates as follows (see, e.g., \cite[pp. $166-167$]{GHL}). Choose an orthonormal basis $\{u, e_2,\ldots,e_n\}$ on $T_PM$ and let $c(t) =\text{\rm Exp}_P(tu)$ be a geodesic curve. The Jacobian fields $\{Y_i(t)\}_{i=2}^n$ satisfy $Y_i(0) =0$, $Y_i'(0) =e_i$, so that the density function can be given by
\[
J(u,t) = t^{1-n} \sqrt{\text{\rm det}(\la Y_i(t), Y_j(t)\ra)},\quad t >0.
\]
We note that $J(u,t)$ does not depend on $\{e_2,\ldots,e_n\}$ and $J(u,t)\in C^\infty(T_PM \setminus\{0\})$ by the definition of $J(u,t)$. Moreover, if we set $J(u,0)\equiv 1$ then $J(u,t) \in C(T_PM)$ and has the following asymptotic expansion
\begin{equation}\label{eq:density}
J(u,t) = 1 + O(t^2)
\end{equation}
as $t\to 0$ since $Y_i(t)$ has the asymptotic expansion (see, e.g., \cite[p. $169$]{GHL})
\[
Y_i(t) = t e_i -\frac{t^3}6 R(c'(t), e_i) c'(t) + o(t^3),
\]
as $t\to 0$, where $R(\cdot, \cdot)$ is the curvature tensor on $M$.

From the definition of $J(u,t)$, we have the following polar coordinate on $M$
\begin{equation}\label{eq:polar}
\int_M f dV =\int_{S^{n-1}}  \int_0^{\rho(u)}f(\text{\rm Exp}_P(tu)) J(u,t) t^{n-1} dt du, 
\end{equation}
where $du$ denotes the canonical measure of the unit sphere of $T_PM$ and $\rho(u)$ denotes the distance to the cut-locus in the direction $u$ (see \cite[Section 2.C.7]{GHL} for the definition of cut-locus). Moreover, the Laplacian of the distance function $d_P(x)$ has the following expansion via the function $J(u,t)$ (see, e.g., \cite[Section 4.B.2]{GHL})
\begin{equation}\label{eq:Lapofdist}
\De_g d_P(x) = \frac{n-1}{d_P(x)} + \frac{J'(u_x,d_P(x))}{J(u_x,d_P(x))},\quad \rho >0,
\end{equation}
for any point $x\not =p$ which is not on the cut-locus of $P$, where $u_x$ is the unique direction in $S^{n-1} \subset T_PM$ such that $x = \text{Exp}_P(d_P(x) u_x)$ and  $J'(u,t) = \frac{\pa J(u,t)}{\pa t}$ with $t < \rho(u)$. Therefore, for any radial function $f(d_P)$ on $M$, we have
\begin{equation}\label{eq:Lapofradial}
\De_g f(d_P(x)) = f''(d_P(x)) + \lt(\frac{n-1}{d_P(x)} + \frac{J'(u_x,d_P(x))}{J(u_x,d_P(x))}\rt) f'(d_P(x)),
\end{equation}
for any point $x\not =P$ which is not on the cut-locus of $P$. Note that if the sectional curvature $K_M$ is constant then $J(u,t)$ depends only on $t$. We denote by $J_b(t)$ the corresponding density function if $K_M\equiv -b$ for some $b\geq 0$. Hence
\[
J_b(t) = \begin{cases}
1 &\mbox{if $b =0$}\\
\lt(\frac{\sinh(\sqrt{b}t)}{\sqrt{b} t}\rt)^{n-1} &\mbox{if $b>0$}.
\end{cases}
\]
For $b \geq 0$, we consider the function $\text{\bf ct}_b :(0,\infty) \to \R$ defined by
\[
\text{\bf ct}_b(t) = 
\begin{cases}
\frac1 t&\mbox{if $b=0$}\\
\sqrt{b} \coth(\sqrt{b} t)&\mbox{if $b>0$},
\end{cases}
\]
and the function $\text{\bf D}_b :[0,\infty) \to \R$ defined by
\[
\text{\bf D}_b(t) = 
\begin{cases}
0 &\mbox{if $t =0$}\\
t \text{\bf ct}_b(t) -1 &\mbox{if $t>0$}.
\end{cases}
\]
Clearly, we have $\text{\bf D}_b \geq 0$.

In our proofs below, we will need the following Bishop--Gunther comparison theorem (see, e.g., \cite[p. $172$]{GHL} for its proof) which says that if the sectional curvature $K_M$ on $M$ satisfies $K_M \leq -b$ for some $b\geq 0$ then
\begin{equation}\label{eq:Bishop}
\frac{J'(u,t)}{J(u,t)} \geq \frac{J_b'(t)}{J_b(t)} = \frac{n-1}t \bD_b(t),\qquad t>0.
\end{equation}
In particular, the function $t \to J(u,t)$ is non-decreasing for any $u\in S^{n-1}$ hence the function $\rho \to \frac{\Vol_g(B(x,\rho))}{\rho^n}$ is non-decreasing. Combining this together with \eqref{eq:locallimit}, we obtain
\begin{equation}\label{eq:comparevolume}
\Vol_g(B(x,\rho)) \geq \om_n \rho^n, \quad \forall x\in M,\, \rho >0. 
\end{equation}
Furthermore, equality holds in \eqref{eq:comparevolume} then $B(x,\rho)$ is isometric to $B_\rho(0)$ (see, e.g., \cite[Theorem $III.4.2$]{Chavel}).

If $(M,g)$ has non-negative Ricci curvature, then the function $\rho \to \frac{\Vol_g(B(x,\rho))}{\rho^n}$ is non-increasing. Combining this together with \eqref{eq:locallimit}, we obtain
\begin{equation}\label{eq:comparevolumea}
\Vol_g(B(x,\rho)) \leq \om_n \rho^n, \quad \forall x\in M,\, \rho >0. 
\end{equation}
Furthermore, equality holds in \eqref{eq:comparevolume} then $B(x,\rho)$ is isometric to $B_\rho(0)$ (see, e.g., \cite[Theorem $III.4.4$]{Chavel}).


\section{The sharp CKN inequalities: Proof of Theorems \ref{myCKN} and \ref{CKNonCH}}
This section is devoted to proved the CKN inequalities in Cartan--Hadamard manifolds. Let $(M,g)$ be an $n-$dimensional Cartan--Hadarmad manifold with $n \geq 2$. Suppose that $K_M \leq -b$ with $b\geq 0$. Furthermore, we will prove a quantitative version of the CKN inequalities by adding the a non-negative remainder terms concerning to the upper bound of $K_M$. Let $P\in M$ be fixed, for $p> 1$ and $\xi, \eta \in T_PM$, we denote
\begin{equation}\label{eq:Remainder}
R_p(\xi, \eta) = \frac1p |\eta|^p + \frac{p-1}p |\xi|^p -|\xi|^{p-2} \la \xi, \eta\ra.
\end{equation}
By the convexity of $\xi \to |\xi|^p$ we see that $R_p(\xi,\eta) \geq 0$ with equality if and only if $\xi = \eta$. Furthermore, we can see that
\[
R_p(\xi, \eta) = (p-1)\int_0^1 |t\xi + (1-t)\eta|^{p-2} t dt |\xi -\eta|^2.
\]
For $p >1$, we always use $p'$ to denote the conjugate exponent of $p$, i.e., $p' =p/(p-1)$. Let $g =(g_1,\ldots,g_m)$, $m\geq 1$ be such that each function $g_i$ is in $L^{p'}(M)$, we define the new function $\mathcal D_p (g)$ on $M$ by
\[
\mathcal D_p(g)(x) =
\begin{cases}
0 &\mbox{if $g(x) =0$}\\
\frac{g(x)}{|g(x)|} |g(x)|^{\frac1{p-1}}&\mbox{if $g(x) \not =0$}
\end{cases}
\]
Then $|\mathcal D_p(g)| \in L^p(M)$. With these notation, we have the following equality which improves H\"older inequality: let $f =(f_1,\ldots,f_m) \in L^p(M)$ and $g =(g_1,\ldots,g_m) \in L^{p'}(M)$ be non-identically zero functions, then it holds
\begin{equation}\label{eq:refineHolder}
\int_M f \cdot g\, dV_g = \|f\|_p \|g\|_{p'}\lt(1 -\int_M R_p\lt(\frac{f}{\|f\|_p},\frac{\mathcal D_p(g)}{\|g\|_{p'}^{\frac1{p-1}}}\rt) dV_g\rt),
\end{equation}
here we use the notation $\|f\|_p= (\int_{M} |f|^p dV_g)^{1/p}$ and for a subset $\Om \subset M$, we shall denote $\|f\|_{p,\Om} = (\int_\Om |f|^p dV_g)^{1/p}$.

With these notation, we can state our first main result of this section as follows.
\begin{theorem}\label{Mainequality}
Suppose $n\geq 2$ and $p,q,r,\al,\be, \ga$ satisfy the conditions \eqref{eq:integrabilitycond} and \eqref{eq:balancecond}. Let $(M,g)$ be an $n-$dimensional Cartan--Hadamard manifold. Then the following equality holds true for any function $f \in C_0^\infty(M)$
\begin{align}\label{eq:mainequality}
\int_M \frac{|f|^r}{d_P^{\ga r}} dV_g &= \frac{r}{n-\ga r} \lt\|\frac{\pa_\rho f}{d_P^\al}\rt\|_{p,\sf}\lt\|\frac{|f|^{r-1}}{d_P^{\be/{p'}}}\rt\|_{p',\sf}\notag\\
&\quad -\frac{r}{n-\ga r}\lt\|\frac{\pa_\rho f}{d_P^\al}\rt\|_{p,\sf}\lt\|\frac{|f|^{r-1}}{d_P^{\be/{p'}}}\rt\|_{p',\sf} \times \notag\\
&\qquad\qquad \quad\qquad \times \int_{\sf}R_P\lt(\frac{d_P^{-\al}\pa_\rho f}{\|d_P^{-\al} \pa_\rho f\|_{p,\sf}},\frac{\mathcal D_p(-d_P^{-\frac\be{p'}} f|f|^{r-2})}{\|d_P^{-\frac\be{p'}}|f|^{r-1}\|_{p',\sf}^{\frac1{p-1}}}\rt) dV_g\notag\\
&\quad -\frac1{n-\ga r} \int_{\sf} \frac{|f|^r}{d_P(x)^{\ga r}} \frac{d_P(x) J'(u_x,d_P(x))}{J(u_x,d_P(x))} dV_g,
\end{align}
where $u_x$ denotes the unique unit vector in $T_PM$ such that $x= \text{\rm Exp}_P(d_P(x)u_x)$.
\end{theorem}
Since $K_M \leq 0$ then $J'\geq 0$ by \eqref{eq:Bishop}. In the other hand $R_p \geq 0$. Hence \eqref{eq:mainequality} implies the CKN inequalities in Theorems \ref{myCKN} and \ref{CKNonCH} since
\[
\int_{\sf} \frac{|\pa_\rho f|^p}{d_P^{\al p}} dV_g = \int_M \frac{|\pa_\rho f|^p}{d_P^{\al p}} dV_g,
\]
and if $r>1$
\[
\int_{\sf} \frac{|f|^{\frac{p(r-1)}{p-1}}}{d_P(x)^\beta} dV_g = \int_{M} \frac{|f|^{\frac{p(r-1)}{p-1}}}{d_P(x)^\beta} dV_g.
\]
Especially, if $K_M \leq -b$ for some $b \geq 0$, we obtain from \eqref{eq:mainequality} and Gauss lemma, the following quantitative CKN inequalities
\begin{corollary}\label{QuantCKN}
Suppose the assumptions in Theorem \ref{Mainequality} and $K_M \leq -b$ for some $b\geq 0$. Then the following inequalities holds for any function $f\in C_0^\infty(M)$,
\begin{equation}\label{eq:quantCKN}
\int_{M} \frac{|f|^r}{d_P(x)^{\ga r}} \lt(1+ \frac{n-1}{n-\ga r} \bD_b(d_P(x))\rt) dV_g \leq \frac{r}{n-\ga r} \lt\|\frac{\pa_\rho f}{d_P^\al}\rt\|_{p}\lt\|\frac{|f|^{r-1}}{d_P^{\be/{p'}}}\rt\|_{p',\sf},
\end{equation}
and
\begin{equation}\label{eq:quantCKNfull}
\int_{M} \frac{|f|^r}{d_P(x)^{\ga r}} \lt(1+ \frac{n-1}{n-\ga r} \bD_b(d_P(x))\rt) dV_g \leq \frac{r}{n-\ga r} \lt\|\frac{\na_g f}{d_P^\al}\rt\|_{p}\lt\|\frac{|f|^{r-1}}{d_P^{\be/{p'}}}\rt\|_{p',\sf},
\end{equation}
\end{corollary}
The case $p =r$, $\beta = p + \de$ and $\al = \frac\de p$ with $\de < n-p$, Corollary \ref{QuantCKN} implies the following quantitative weighted Hardy inequalities on Cartan--Hadamard manifolds which are recently prove by the author \cite{Nguyen2017}
\begin{equation}\label{eq:QH}
\int_M \frac{|f|^p}{d_P^{p+ \de}}\lt(1+ \frac{p(n-1)}{n-p-\de} \bD_b(d_P(x))\rt)dV_g \leq \lt(\frac{p}{n-p-\de}\rt)^p \int_M \frac{|\pa_\rho f|^p}{d_P^\de} dV_g,
\end{equation}
and
\begin{equation}\label{eq:QHfull}
\int_M \frac{|f|^p}{d_P^{p+ \de}}\lt(1+ \frac{p(n-1)}{n-p-\de} \bD_b(d_P(x))\rt)dV_g \leq \lt(\frac{p}{n-p-\de}\rt)^p \int_M \frac{|\na_g f|^p}{d_P^\de} dV_g.
\end{equation}
These inequalities improve the weighted Hardy inequalities on Cartan--Hadamard manifolds due to Yang, Su, and Kong \cite{YSK}. The case $p=2$ and $\de =0$, the inequality \eqref{eq:QHfull} was proved by Krist\'aly in \cite{K2}. We refer readers to \cite{Nguyen2017} for more results about the critical Hardy, and Rellich type inequalities on Cartan--Hadamard inequality. The last comment in the case $(M,g)$ having constant sectional curvature, i.e., $K_M = -b$ for some $b\geq 0$ is that the extremal for \eqref{eq:quantCKN} and \eqref{eq:quantCKNfull} exist if one of the conditions (i)--(v) in Theorem \ref{myCKN} holds. Moreover, a family of extremal is given by the same family of extremal in the corresponding case in Euclidean space with $|x|$ being replaced by $d_P(x)$. Indeed, in this case, we always have
\[
R_P\lt(\frac{d_P^{-\al}\pa_\rho f}{\|d_P^{-\al} \pa_\rho f\|_{p,\sf}},\frac{\mathcal D_p(-d_P^{-\frac\be{p'}} f|f|^{r-2})}{\|d_P^{-\frac\be{p'}}|f|^{r-1}\|_{p',\sf}^{\frac1{p-1}}}\rt) =0
\]
if $f$ has such form, and $\rho J'(u,\rho)/J(u,\rho) = (n-1)\bD_b(\rho)$. Therefore equality holds true in \eqref{eq:mainequality}. Furthermore, the proof of Theorems \ref{Maintheorem1} and \ref{Maintheorem3} below can be applied to prove a rigidity results for Cartan--Hadamard manifolds $(M,g)$ with $K_M\leq -b$ for some $b\geq 0$ such that extremal for the inequality \eqref{eq:quantCKNfull} exists. Such a manifold should be isometric to a manifold of constant sectional curvature $-b$. Evidently, the case $b=0$ is considered in Theorems \ref{Maintheorem1} and  \ref{Maintheorem3}.

Let us prove Theorem \ref{Mainequality}

\begin{proof}[Proof of Theorem \ref{Mainequality}]
The proof is simple by using integration by parts. Indeed, let $f\in \C_0^\infty(M)$, by abusing notation we still denote by $\sf$ for the preimage of support of $f$ in $T_PM$ and $f(tu)$ for $f(\text{\rm Exp}_P(tu))$. Using polar coordinate \eqref{eq:polar}, we have
\begin{align*}
\int_M \frac{|f|^r}{d_P(x)^{\ga r}} dV_g & = \int_{S^{n-1}} \int_{\sf \cap [0,\infty)} |f(t u)|^r t^{n-\ga r -1} J(u,t) dt du\\
&=\frac1{n-\ga r} \int_{S^{n-1}} \int_{\sf \cap [0,\infty)} |f(t u)|^r (t^{n-\ga r})' J(u,t) dt du.
\end{align*}
Using integration by parts and the assumption $n-\ga r >0$ , we get
\begin{align*}
\int_M \frac{|f|^r}{d_P(x)^{\ga r}} dV_g& = -\frac{r}{n-\ga r} \int_{S^{n-1}} \int_{\sf \cap [0,\infty)} |f(t u)|^{r-2} f(t u)\pa_\rho f(t u) t^{n-\ga r} J(u,t) dt du\\
&=-\frac1{n-\ga r} \int_{S^{n-1}} \int_{\sf \cap [0,\infty)} |f(tu)|^r t^{n-\ga r} J'(u,t) dt du.
\end{align*}
Using again polar coordinate \eqref{eq:polar} and the condition \eqref{eq:balancecond}, we arrive
\begin{align*}
\int_M \frac{|f|^r}{d_P(x)^{\ga r}} dV_g& = -\frac{r}{n-\ga r}\int_{\sf} \frac{|f|^{r-2} f}{d_P^{\be /{p'}}} \frac{\pa_\rho f}{d_P^\al} dv_g  -\frac1{n-\ga r} \int_{\sf} \frac{|f|^r}{d_P^{\ga r}} \frac{d_P(x) J'(u_x,d_P(x))}{J(u_x, d_P(x))} dV_g.
\end{align*}
Now, using \eqref{eq:refineHolder}, we obtain our desired equality \eqref{eq:mainequality}.
\end{proof}

We are now ready to prove Theorems \ref{myCKN} and \ref{CKNonCH}.

\begin{proof}[Proof of Theorem \ref{myCKN}]
The inequalities are trivial by remarks after Theorem \ref{Mainequality}. The sharpness of constant $r/(n-\ga r)$ is immediately checked by functions given in Theorem \ref{myCKN} corresponding to the conditions (i) -- (v). Indeed, in these cases, we have 
\[
R_P\lt(\frac{d_P^{-\al}\pa_\rho f}{\|d_P^{-\al} \pa_\rho f\|_{p,\sf}},\frac{\mathcal D_p(-d_P^{-\frac\be{p'}} f|f|^{r-2})}{\|d_P^{-\frac\be{p'}}|f|^{r-1}\|_{p',\sf}^{\frac1{p-1}}}\rt) = 0,
\]
on $\sf$ and $J'\equiv 0$. This finishes the proof of Theorem \ref{myCKN}.
\end{proof}

\begin{proof}[Proof of Theorem \ref{CKNonCH}]
The inequalities are trivial by remarks after Theorem \ref{Mainequality}. Let us verify the sharpness of the constant $\frac{r}{n-\ga r}$. We known from Theorem \ref{myCKN} that if one of the conditions (i) -- (v) in Theorem \ref{myCKN} holds true, then
\[
\frac{r}{n-\ga r} = \sup_{f\in C_0^\infty(\R^n)} \frac{\int_{\R^n} |x|^{-\ga r} |f|^r dx}{\lt(\int_{\R^n} |x|^{-\al p} |\pa_r f|^p dx\rt)^{\frac1p} \lt(\int_{\sf} |x|^{-\beta} |f|^{p'(r-1)} dx\rt)^{\frac1{p'}}}.
\]
Furthermore, we can assume that the supremum is taken on non-negative radial functions. Indeed, from Theorem \ref{myCKN}, we see that the extremal of CKN inequalities contain non-negative radial functions by taking $\lam \equiv\text{\rm const}$, and hence we can approximate these functions by non-negative radial functions in $C_0^\infty(\R^n)$. For any $\ep >0$, we can chose a non-negative radial function $f_\ep \in C_0^\infty$ such that 
\[
\frac{r}{n-\ga r} -\ep\leq \sup_{f_\ep\in C_0^\infty(\R^n)} \frac{\int_{\R^n} |x|^{-\ga r} |f_\ep|^r dx}{\lt(\int_{\R^n} |x|^{-\al p} |\pa_r f_\ep|^p dx\rt)^{\frac1p} \lt(\int_{\sf_\ep} |x|^{-\beta} |f_\ep|^{p'(r-1)} dx\rt)^{\frac1{p'}}}.
\]
For $\de>0$, denote $f_{\ep,\de}(x) = f_\ep(x/\de)$ then $\sf_{\ep,\de}=\de \sf_\ep$. The scaling invariant of CKN inequalities implies that
\begin{equation}\label{eq:ss}
\frac{r}{n-\ga r} -\ep\leq \sup_{f_\ep\in C_0^\infty(\R^n)} \frac{\int_{\R^n} |x|^{-\ga r} |f_{\ep,\de}|^r dx}{\lt(\int_{\R^n} |x|^{-\al p} |\pa_r f_{\ep,\de}|^p dx\rt)^{\frac1p} \lt(\int_{\sf_{\ep,\de}} |x|^{-\beta} |f_\ep|^{p'(r-1)} dx\rt)^{\frac1{p'}}}
\end{equation}
for any $\de >0$.

Suppose that $f_\ep(x) = \vphi(|x|)$ for some function $\vphi$ with $\vphi\equiv 0$ on $[a,\infty)$ for some $a>0$. Define the functions $F_\de$ on $M$ by $F(x) = \vphi(d_P(x)/\de)$. Evidently, $\text{\rm supp} F_\de \subset B(P,a\de)$. Using polar coordinate \eqref{eq:polar}, we have
\begin{align*}
\int_{M} \frac{|F_\de|^r}{d_P^{\ga r}} dV_g&= \int_{S^{n-1}} \int_{\de \text{\rm supp}\vphi} \vphi(t/\de)^r t^{n-\ga r -1} J(u,t) dt du\\
&= (1+ O(\de^2))\int_{S^{n-1}} \int_{\de \text{\rm supp}\vphi} \vphi(t/\de)^r t^{n-\ga r -1} J(u,t) dt du\\
&=(1+ O(\de^2)) \int_{\R^n} \frac{|f_{\ep,\de}|^r}{|x|^{\ga r}} dx,
\end{align*}
here we use \eqref{eq:density}. Similarly, we get
\[
\int_M \frac{|\pa_\rho F_\de|^p}{d_P^{\al p}} dV_g = (1+ O(\de^2)) \int_{\R^n} \frac{|\pa_r f_{\ep,\de}|^p}{|x|^{\al p}}dx,
\]
and
\[
\int_{\text{\rm supp}F_\de} \frac{|F_\de|^{\frac{p(r-1)}{p-1}}}{d_P^\beta}dV_g =(1+ O(\de^2))\int_{\sf_{\ep,\de}} \frac{|f_{\ep,\de}|^{\frac{p(r-1)}{p-1}}}{|x|^{\beta}} dx,
\]
here we use the fact $f$ is radial. Combining these three equalities together with \eqref{eq:ss}, we obtain
\[
\liminf_{\de \to 0} \frac{\int_{M} \frac{|F_\de|^r}{d_P^{\ga r}} dV_g}{\lt(\int_M \frac{|\pa_\rho F_\de|^p}{d_P^{\al p}} dV_g\rt)^{\frac1p} \lt(\int_{\text{\rm supp}F_\de} \frac{|F_\de|^{\frac{p(r-1)}{p-1}}}{d_P^\beta}dV_g\rt)^{\frac{p-1}p} } \geq \frac{r}{n-\ga r} - \ep,
\]
for any $\ep >0$. This implies the sharpness of $r/(n-\ga r)$.
\end{proof}


\section{Rigidity results on Cartan--Hadamard manifolds: Proof of Theorems \ref{Maintheorem1} and \ref{Maintheorem3}}
In this section, we give the proof of Theorems \ref{Maintheorem1} and \ref{Maintheorem3}. The main ingredients in our proofs is Theorem \ref{Mainequality}, the Gauss lemma and the explicit solutions of several ordinary differential equations related to the Euler--Lagrange equations of the extremal for the CKN inequalities. We first prove Theorem \ref{Maintheorem1}.

\begin{proof}[Proof of Theorem \ref{Maintheorem1}]
Let $f$ be an extremal which is not identically zero for $\CKNP$. Evidently, $|f|$ also is an extremal for $\CKNP$. Hence we can assume that $f$ is non-negative. Moreover, by Theorem \ref{CKNonCH} (more precisely, inequality \eqref{eq:posexponentCH}), we must have
\[
\int_M \frac{|\pa_\rho f|^p}{d_P(x)^{\al p}} dV_g \geq \int_M \frac{|\na_ g f|^p}{d_P(x)^{\al p}} dV_g,
\]
which implies $|\pa_\rho| = |\na_g f|$ by Gauss lemma. Therefore, $f$ is radial function, that is, $f$ depends only on $d_P$ or $f =\vphi(d_P)$ with $\vphi :[0,\infty) \to [0,\infty)$. By \eqref{eq:mainequality}, we must have 
\[
R_P\lt(\frac{d_P^{-\al}\pa_\rho f}{\|d_P^{-\al} \pa_\rho f\|_{p,\sf}},\frac{\mathcal D_p(-d_P^{-\frac\be{p'}} f|f|^{r-2})}{\|d_P^{-\frac\be{p'}}|f|^{r-1}\|_{p',\sf}^{\frac1{p-1}}}\rt) =0
\]
and $J'(u_x, d_P(x)) =0$ on $\sf$. The first condition is equivalent to
\[
\pa_\rho f = -c f^{\frac{r-1}{p-1}} d_P^{\al -\frac\beta p}
\]
on $\sf$ for some $c>0$. Writing $f$ as $\vphi(d_P(x))$, the previous equation is equivalent to
\begin{equation}\label{eq:ODE}
\vphi'(t) = -c \vphi(t)^{\frac{r-1}{p-1}} t^{\al -\frac\beta p},
\end{equation}
on $\{\vphi >0\}$. Since $f$ is not identically zero, then $\vphi(0) >0$. The equation \eqref{eq:ODE} has unique solution
\[
\vphi(t) = \lt(\vphi(0)^{\frac{p-r}{p-1}}+ c\frac{r-p}{p-1}\frac{t^{1+ \al -\frac\beta p}}{1+ \al -\frac\beta p}\rt)^{\frac{p-1}{p-r}}
\]
if $r > p > 1$ and
\[
\vphi(t) =\vphi(0) \exp\lt(-\frac{c}{1+ \al -\frac\be p} t^{1+ \al -\frac\be p}\rt)
\]
if $r =p$. Hence $\sf = M$, and the condition $J'(u_x,d_P(x)) =0$ on $\sf$ translates to $J'(u,t) =0$ for any $t>0$ and for each fixed $u\in S^{n-1}$. This implies $J(u,t) \equiv 1$ for any $t>0$ and $u\in S^{n-1}$. Hence $M$ is isometric to $\R^n$.
\end{proof}
Theorem \ref{Maintheorem3} is proved by the same way.

\begin{proof}[Proof of Theorem \ref{Maintheorem3}]
Obviously, as in proof of Theorem \ref{Maintheorem1}, we can assume that the extremal $f$ is non-negative and $f$ is radial function, that is, $f$ depends only on $d_P$ or $f =\vphi(d_P)$ with $\vphi :[0,\infty) \to [0,\infty)$. Moreover, we must have $J'(u_x,d_P(x)) =0$ on $\sf$ and
\[
\pa_\rho f = -c |f|^{\frac{r-p}{p-1}} f d_P^{\al -\frac\be p}
\]
on $\sf$ for some $c >0$. Consequently, $\vphi$ satisfies 
\begin{equation}\label{eq:ODEvarphi}
\vphi'(t) = -c \vphi(t)^{\frac{r-1}{p-1}} t^{\al -\frac\beta p}
\end{equation}
on $\{\vphi > 0\}$. Hence $\vphi$ is strict increasing on $\{\vphi >0\}$. If $0< r < p, r\not=1$ and $1+ \al -\frac\beta p > 0$ then \eqref{eq:ODEvarphi} has only solution of the form
\[
\vphi(t) = \lt(\lam -c\frac{p-r}{p-1} \frac{t^{1+ \al -\frac\be p}}{1+ \al -\frac\be p} \rt)_+^{\frac{p-1}{p-r}}, \quad \lam >0.
\]
If $0< r < p, r\not =1$ and $1+ \al -\frac\beta p =0$ then \eqref{eq:ODEvarphi} has only solution
\[
\vphi(t) = \lt(\lam -c \frac{p-r}{p-1} \ln t\rt)_+^{\frac{p-1}{p-r}},\quad \lam \in \R.
\]
If $0< r < p, r\not=1$, $1 + \al -\frac\beta p < 0$ and $n-\beta + (1+ \al -\frac{\beta}p) \frac{p(r-1)}{p-r}>0$ then \eqref{eq:ODEvarphi} has only solution
\[
\vphi(t) = \lt(-c\frac{p-r}{p-1}\frac{t^{1+\al -\frac\be p}}{1+ \al -\frac \be p} -\lam\rt)_+^{\frac{p-1}{p-r}},\qquad \lam \in \R.
\]
To ensure $\int_{\sf} d_P^{-\beta} f^{\frac{p(r-1)}{p-1}} dV_g < \infty$, we must have $\lam >0$.

The form of function $\vphi$ above shows that $f$ has compact support and there is $r_P >0$ such that $\{f>0\}= B(P,r_P)$. Therefore, the condition $J'(u_x,d_P(x)) =0$ on $\sf$ is equivalent to $J'(u,t) = 0$ for any $t< r_P$ and $u\in S^{n-1}$. Thus $J(u,t) \equiv 1$ for any $t < r_P$ and $u\in S^{n-1}$ which implies $\Vol_g(B(P,r_P)) = \om_n r_P^n$ by polar coordinate \eqref{eq:polar}. Hence $B(P,r_P)$ is isometric to $B_{r_P}(0)$ (see, e.g., \cite[Theorem $III.4.2$]{Chavel}).
\end{proof}

An immediate consequence of Theorem \ref{Maintheorem3} is that if the constant $r/(n-\ga r)$ is attained by an extremal which is not identically zero in $\CKNP$ for a point $P \in M$, then the sectional curvature at $P$ vanishes. Therefore, if if the constant $r/(n-\ga r)$ is attained by an extremal which is not identically zero in $\CKNP$ for any point $P \in M$, then $M$ is flat, i.e., the sectional curvature vanishes at any point of $M$. 
\begin{corollary}
Suppose the assumptions in Theorem \ref{Maintheorem3}. Then the following statements are equivalent:
\begin{description}
\item (a) $\frac{r}{n-\ga r}$ is achived by an extremal which is not identically zero in $\CKNP$ for any point $P \in M$.
\item (b) M is flat.
\end{description}
\end{corollary}

\section{Rigidity results on manifolds with non-negative Ricci curvature: Proof of Theorems \ref{Maintheorem2} and \ref{Maintheorem4}}
This section is devoted to prove the rigidity results in Theorem \ref{Maintheorem2} and \ref{Maintheorem4}. The main idea in our proof goes back to the work of Ledoux on sharp Sobolev inequality \cite{Ledoux} and then developed by several works \cite{Druet,Bakry,K1,K2,KO,Xia,Xiasobolev,CarXia,XiaGN,Minerbe,Cheeger,LiWang}. The crucial ingredient is the explicit form of extremal in the Euclidean spaces. Exploiting this form of extremal, we define a new function in $(M,g)$ which depends only on $d_P$ and then applying the CKN inequality to obtain a differential inequality. Using this differential inequality, we obtain the equality in the Bishop--Gromov volume comparison theorem, and hence obtain the desired result. It is worthy to emphasize here that the extremal of CKN inequalities considered in Theorem \ref{Maintheorem4} are compactly supported functions. This arises several difficulties in the proof of Theorem \ref{Maintheorem4}. 

\begin{proof}[Proof of Theorem \ref{Maintheorem2}]
The implications $(c) \Rightarrow (b) \Rightarrow (a)$ are trivial by Corollary \ref{myCKNfull}. It remains to verify $(a) \Rightarrow (c)$. For $\lam >0$, define the function
\[
T(\lam) = \int_{\R^n} e^{-p \lam |x|^{1+ \al -\frac\beta p}} |x|^{-\ga p}dx.
\]
We can check that
\[
T(\lam) = \lam^{-\frac{n-\ga p}{1+ \al -\frac\beta p}} p^{-\frac{n-\ga p}{1+ \al -\frac\beta p}} \frac{n\om_n}{1+ \al -\frac\beta p} \Gamma\lt(\frac{n-\ga p}{1+ \al -\frac\beta p}\rt),
\]
and hence $T$ satisfies the equation
\begin{equation}\label{eq:ODET}
-\lam T'(\lam) = \frac{n-\ga p}{1+ \al -\frac\beta p} T(\lam), \quad \lam >0.
\end{equation}
Let $P\in M$ be fixed. Since $\CKNP$ holds, then $(M,g)$ cannot be compact. For $\lam >0$, we define
\[
u_\lam (x) = e^{-\lam d_P(x)^{1+ \al -\frac\beta p}},\quad \lam >0.
\]
By a simple approximation procedure, we can apply $\CKNP$ to function $u_\lam$ and then obtain the following inequality (note that $|\na d_P| =1$)
\begin{equation}\label{eq:useCKN}
\int_M e^{-p\lam d_P^{1+ \al -\frac\beta p}} d_P^{-\ga p} dV_g \leq \lam \frac{p(1+ \al -\frac\be p)}{n-\ga p}\int_M e^{-p\lam d_P^{1+ \al -\frac\beta p}} d_P^{-\beta} dV_g, \quad \lam >0.
\end{equation}
Define 
\[
F(\lam) = \int_M e^{-p \lam d_P^{1+ \al -\frac\beta p}} d_P^{-\ga p} dV_g.
\]
Using Bishop--Gromov comparison theorem, we can easily check that $0 < F(\lam) < \infty$ for any $\lam >0$ and $F$ is differentiable on $(0,\infty)$. Moreover, we can compute that
\[
F'(\lam) = -p\int_M e^{-p \lam d_P^{1+ \al -\frac\beta p}} d_P^{-\beta} dV_g,
\]
and hence
\begin{equation}\label{eq:ODEF}
-\lam F'(\lam) \geq \frac{n-\ga p}{1+ \al -\frac{\be}p} F(\lam),\quad \lam >0.
\end{equation}
Combining \eqref{eq:ODET} and \eqref{eq:ODEF}, we get $(F/T)' \leq 0$ hence the function $\lam \to \frac{F(\lam)}{T(\lam)}$ is non-increasing. In particular, for any $\lam >0$,
\begin{equation}\label{eq:limitinfinity}
\frac{F(\lam)}{T(\lam)} \geq \lim_{\lam \to \infty} \frac{F(\lam)}{T(\lam)}.
\end{equation}

We next make an estimate of $F(\lam)$ for $\lam >0$ large enough. A traditional way is to use the layer cake representation
\[
F(\lam) = \int_0^\infty\Vol_g \lt(\{x\in M\, :\, e^{-p \lam d_P^{1+ \al -\frac\beta p}} d_P^{-\ga p} > t\}\rt) dt,
\]
and then making the change of variable $t= e^{-p\lam  s^{1+ \al -\be/p}}s^{-\ga p}$. It seems that this argument does not work for $\ga < 0$ since the function $s \to e^{-p\lam  s^{1+ \al -\be/p}}s^{-\ga p}$ is not decreasing monotone on $[0,\infty)$. The same situation also appeared in the proof of Theorem $1.3$ of Xia \cite{Xia}. Instead of using the layer cake representation, we will use the polar coordinate \eqref{eq:polar}. For $0< \ep < \inf\{\rho(u)\, :\, u\in S^{n-1}\}$ we have
\begin{align*}
F(\lam) &=  \int_{S^{n-1}}\int_0^{\rho(u)} e^{-p \lam t^{1+ \al -\frac\be p}} J(u,t) t^{n-\ga p -1} dt du\\
&\geq \int_0^\ep\int_{S^{n-1}} e^{-p \lam t^{1+ \al -\frac\be p}} J(u,t) t^{n-\ga p -1} dt du\\
&=(1+ O(\ep)^2) n\om_n \int_0^\ep e^{-p \lam t^{1+ \al -\frac\be p}} t^{n-\ga p -1} dt du\\
&=(1+ O(\ep)^2) \lam^{-\frac{n-\ga p}{1+ \al -\frac\beta p}} p^{-\frac{n-\ga p}{1+ \al -\frac\beta p}} \frac{n\om_n}{1+ \al -\frac\beta p} \int_0^{(p\lam)^{\frac1{1+ \al -\be/p}}\ep} e^{-t} t^{n-\ga r -1} dt.
\end{align*}
Consequently, we obtain
\[
\lim_{\lam \to \infty} \frac{F(\lam)}{T(\lam)} \geq 1+ O(\ep^2).
\]
Letting $\ep \to 0$, we get
\[
\frac{F(\lam)}{T(\lam)} \geq \lim_{\lam \to \infty} \frac{F(\lam)}{T(\lam)} \geq 1.
\]
Hence $F(\lam) \geq T(\lam) $ for any $\lam >0$. In the other hand, for any $u\in S^{n-1}$, we have $J(u,t) \leq 1$ for any $t < \rho(u)$ (see, e.g., \cite[p. $172$, line $5$]{GHL}). Hence, using again the polar coordinate \eqref{eq:polar}, we obtain
\begin{align*}
T(\lam) \leq F(\lam) &= \int_{S^{n-1}}\int_0^{\rho(u)}  e^{-p\lam t^{1+ \al -\frac\beta p}}t^{n-\ga p -1} J(u,t) dt du\\
&\leq \int_{S^{n-1}}\int_0^{\rho(u)}  e^{-p\lam t^{1+ \al -\frac\beta p}}t^{n-\ga p -1} dt du\\
&\leq \int_{S^{n-1}}\int_0^{\infty}  e^{-p\lam t^{1+ \al -\frac\beta p}}t^{n-\ga p -1} dt du\\
&=T(\lam).
\end{align*}
Consequently, we get $F(\lam) =T(\lam)$ for any $\lam >0$. Thus all inequalities in the previous estimates are equalities. This implies that $\rho(u) =\infty$ for almost $u\in S^{n-1}$ and $J(u,t) \equiv 1$ for any $t < \rho(u)$. This together with the polar coordinate \eqref{eq:polar} implies $\Vol_g(B(P,r)) = \om_n r^n$ for any $r>0$. The equality condition in Bishop--Gromov volume comparison theorem (see \cite[Theorem $III.4.4$]{Chavel}) implies that $M$ is isometric to $\R^n$
\end{proof}

We next move to the proof of Theorem \ref{Maintheorem4}. We will need the following simple result.
\begin{proposition}\label{key}
Let $(\Om, \mu)$ be a measure space and $f: \Om\to [0,\infty)$ be a measurable function. Given $q \in (0,1)\cup (1,\infty)$ and suppose that
\begin{equation}\label{eq:keygt}
\int_{\{f < \lam\}} (\lam -f)^{q-1} d\mu < \infty
\end{equation}
for any $\lam >0$. Suppose, in addition, that 
\begin{equation}\label{eq:keygt2}
\mu(\{\lam \leq f < \lam +h\}) =\begin{cases}
 O(h) &\mbox{if $q\in (0,1)$}\\
O(1) &\mbox{if $q >1$}
\end{cases}
\end{equation}
for $\lam, h >0$. Then the function $G: \lam \to \int_\Om (\lam -f)_+^q d\mu$ is differentiable on $(0,\infty)$ and
\begin{equation}\label{eq:daohamG}
G'(\lam) = q \int_{\{f< \lam\}} (\lam -f)^{q-1} d\mu.
\end{equation}
\end{proposition}
\begin{proof}
Denote $A_\lam = \{f < \lam\}$. For $h >0$, we have
\begin{equation*}
\frac{G(\lam +h) -G(\lam)}h = \int_{A_\lam} \frac{(\lam +h -f)^q -(\lam -f)^q}h d\mu + \frac1h\int_{\{\lam\leq f< \lam+h\}} (\lam+ h-f)^q d\mu.
\end{equation*}
It is easy to check that 
\[
(a+b)^q - a^q \leq C a^{q-1} b\, \quad a, b>0,
\]
for some constant $C > 0$ is. Combining the previous inequality with \eqref{eq:keygt} and the Lebesgue dominated convergence theorem, we get
\[
\lim_{h\to 0^+} \int_{A_\lam} \frac{(\lam +h -f)^q -(\lam -f)^q}h d\mu =q \int_{\{f< \lam\}} (\lam -f)^{q-1} d\mu.
\]
In the other hand, we have $\frac1h\int_{\{\lam\leq f< \lam+h\}} (\lam+ h-f)^q d\mu \leq h^{q-1} \mu(\{\lam \leq f < \lam+h\})$. The assumption \eqref{eq:keygt2} implies 
\[
\lim_{h\to 0^+} \frac1h\int_{\{\lam\leq f< \lam+h\}} (\lam+ h-f)^q d\mu =0.
\]
Thus, we get 
\begin{equation}\label{eq:daohamphai}
\lim_{h\to 0^+} \frac{G(\lam+ h) -G(\lam)}{h} = q \int_{\{f< \lam\}} (\lam -f)^{q-1} d\mu.
\end{equation}

For $h < 0$, we have
\[
\frac{G(\lam + h) -G(\lam)}h = \int_{A_\lam} \frac{(\lam +h -f)_+^q -(\lam -f)^q}h d\mu.
\]
We claim that
\[
a^q - (a-b)_+^q \leq C a^{q-1} b,\quad a, b>0.
\]
for some constant $C >0$. Indeed, if $b\geq a/2$ then 
\[
a^q -(a-b)_+^q \leq a^q \leq 2 a^{q-1} h.
\]
If $0< b < a/2$, denote $t = b/ a \in (0,1/2)$, we have
\[
a^q -(a-b)^q = a^q(1 - (1-t)^q) \leq C a^q t = Ca^{q-1} b,
\]
for some constant $C >0$, here we use $\lim_{t\to 0^+} \frac{1 -(1-t)^q}t = q$ and $t \in (0,1/2)$. Hence our claim has been proved. Our claim together with \eqref{eq:keygt} and the Lebesgue dominated convergence theorem, we get
\begin{equation}\label{eq:daohamtrai}
\lim_{h\to 0^-} \int_{A_\lam} \frac{(\lam +h -f)^q -(\lam -f)^q}h d\mu =q \int_{\{f< \lam\}} (\lam -f)^{q-1} d\mu.
\end{equation}
Combining \eqref{eq:daohamphai} and \eqref{eq:daohamtrai} finishes the proof of this proposition. 
\end{proof}
\begin{proof}[Proof of Theorem \ref{Maintheorem4}]
The implications $(c) \Rightarrow (b) \Rightarrow (a)$ are trivial by Corollary \ref{myCKNfull}. It remains to proved $(a) \Rightarrow (c)$. In the sequel, we prove only for the case $0< r < p$ and $1+ \al -\frac\beta p >0$. The proof in the other cases is completely similar.

For $\lam >0$, define
\[
T(\lam) = \int_{\R^n}\lt(\lam -|x|^{1+ \al -\frac\beta p}\rt)_+^{r\frac{p-1}{p-r}} |x|^{-r \ga} dx.
\]
A straightforward computation shows that
\[
T(\lam) = \lam^{\frac{n-r \ga}{1+ \al -\frac\be p}+ \frac{r(p-1)}{p-r}} n\om_n \frac1{1+ \al -\frac\be p} B\lt(\frac{r(p-1)}{p-r} + 1, \frac{n-\ga r}{1+ \al -\frac\be p}\rt)
\]
where $B$ denotes the usual beta function. Denote $\de = \frac{n-r \ga}{1+ \al -\frac\be p}+ \frac{r(p-1)}{p-r}$ for short, then it is evident that 
\begin{equation}\label{eq:ODETa}
\lam T'(\lam) =\de T(\lam), \quad \lam >0.
\end{equation}

Let $P\in M$ be fixed. Since $\CKNP$ holds at $P$, then $M$ can not be compact. For $\lam >0$, define the function $u_\lam$ on $M$ by
\[
u_\lam(x) = \lt(\lam -d_P(x)^{1+ \al -\frac\be p}\rt)_+^{\frac{p-1}{p-r}},
\]
and denote
\[
F(\lam) = \int_M \frac{u_\lam^r}{d_P^{\ga r}} dV_g= \int_M \lt(\lam -d_P(x)^{1+ \al -\frac\be p}\rt)_+^{\frac{r(p-1)}{p-r}} d_P(x)^{-\ga r}dV_g.
\]
By a simple approximation argument, we can apply $\CKNP$ for this function $u_\lam$ and obtain for any $\lam >0$,
\begin{equation}\label{eq:BDTforT}
F(\lam) \leq \frac{r}{n-\ga r}\lt(1+ \al -\frac\be p\rt) \frac{p-1}{p-r} \int_{\{d_P < \lam^{\frac1{1+\al -\frac\be p}}\}} \lt(\lam -d_P^{1+ \al -\frac\be p}\rt)^{\frac{p(r-1)}{p-r}} d_P^{-\be} dV_g,
\end{equation}
here we use $|\na d_P| =1$. Since $J(u,t) \leq 1$ for any $t < \rho(u)$ then the assumptions in Proposition \ref{key} satisfy for $f = d_P^{1+ \al -\frac\beta p}$, $(\Om, \mu) =(M,V_g)$ and $q =r(p-1)/(p-r) \in (0,1)\cup (1,\infty)$. As consequence of Proposition \ref{key}, $F$ is differentiable, and
\[
F'(\lam) = \frac{r(p-1)}{p-r} \int_{\{d_P < \lam^{\frac1{1+\al -\frac\be p}}\}}\lt(\lam -d_P^{1+ \al -\frac\be p}\rt)^{\frac{p(r-1)}{p-r}} d_P^{-\ga r} dV_g.
\]
Recall that $\ga r = 1+ \al -\frac\be p + \be$. By an easy computation, we get
\begin{align*}
\lam F'(\lam) & = \frac{r(p-1)}{p-r}\int_{\{d_P < \lam^{\frac1{1+\al -\frac\be p}}\}}\lam \lt(\lam -d_P^{1+ \al -\frac\be p}\rt)^{\frac{p(r-1)}{p-r}} d_P^{-\ga r} dV_g\\
&= \frac{r(p-1)}{p-r}\int_{M} \lt(\lam -d_P^{1+ \al -\frac\be p}\rt)^{\frac{r(p-1)}{p-r}} d_P^{-\ga r} dV_g\\
&\quad + \frac{r(p-1)}{p-r}  \int_{\{d_P < \lam^{\frac1{1+\al -\frac\be p}}\}} \lt(\lam -d_P^{1+ \al -\frac\be p}\rt)^{\frac{p(r-1)}{p-r}} d_P^{-\be} dV_g.
\end{align*}
This together with \eqref{eq:BDTforT} yields
\begin{equation}\label{eq:BDTviphanF}
\lam F'(\lam) \geq \de F(\lam),\quad \lam >0,\, \, \de =\frac{n-r \ga}{1+ \al -\frac\be p}+ \frac{r(p-1)}{p-r}.
\end{equation}
From \eqref{eq:ODETa} and \eqref{eq:BDTviphanF}, we get $F'T -T'F \geq 0$ or the function $\lam \to \frac{F(\lam)}{T(\lam)}$ is non-decreasing on $(0,\infty)$. In the other hand, for $0< \lam^{\frac1{1+ \al -\frac\be p}} < \inf\{\rho(u)\,:\, u\in S^{n-1}\}$, we have by using polar coordinate \eqref{eq:polar}
\begin{align*}
F(\lam) &= \int_{S^{n-1}}\int_0^{\rho(u)} \lt(\lam -t^{1+ \al -\frac\be p}\rt)_+^{\frac{r(p-1)}{p-r}} t^{n-\ga r -1} J(u,t) dt du\\
&= \int_{S^{n-1}}\int_0^{\lam^{\frac1{1+ \al -\frac\be p}}} \lt(\lam -t^{1+ \al -\frac\be p}\rt)_+^{\frac{r(p-1)}{p-r}} t^{n-\ga r -1} J(u,t) dt du\\
&=(1+ O(\lam^{\frac2{1+ \al -\frac\be p}})) n\om_n \int_0^{\lam^{\frac1{1+ \al -\frac\be p}}} \lt(\lam -t^{1+ \al -\frac\be p}\rt)_+^{\frac{r(p-1)}{p-r}} t^{n-\ga r -1} dt\\
&=(1+ O(\lam^{\frac2{1+ \al -\frac\be p}})) T(\lam).
\end{align*}
Thus, we obtain
\[
\lim_{\lam \to 0^+} \frac{F(\lam)}{T(\lam)} =1,
\]
which together with the non-decreasing monotonicity of $\frac{F}T$ immediately implies 
\[
F(\lam) \geq T(\lam),\qquad \lam >0.
\]
In the other hand, we have by using polar coordinate \eqref{eq:polar} and $J(u,t)\leq 1$ for $t < \rho(u)$ that
\begin{align*}
T(\lam) \leq F(\lam) &= \int_{S^{n-1}}\int_0^{\rho(u)} \lt(\lam -t^{1+ \al -\frac\be p}\rt)_+^{\frac{r(p-1)}{p-r}} t^{n-\ga r -1} J(u,t) dt du\\
&=\int_{S^{n-1}}\int_0^{\min\{\rho(u),\lam^{\frac2{1+ \al-\frac\be p}}\}} \lt(\lam -t^{1+ \al -\frac\be p}\rt)_+^{\frac{r(p-1)}{p-r}} t^{n-\ga r -1} J(u,t) dt du\\
&\leq \int_{S^{n-1}}\int_0^{\min\{\rho(u),\lam^{\frac2{1+ \al-\frac\be p}}\}} \lt(\lam -t^{1+ \al -\frac\be p}\rt)_+^{\frac{r(p-1)}{p-r}} t^{n-\ga r -1} dt du\\
&\leq \int_{S^{n-1}}\int_0^{\lam^{\frac2{1+ \al-\frac\be p}}} \lt(\lam -t^{1+ \al -\frac\be p}\rt)_+^{\frac{r(p-1)}{p-r}} t^{n-\ga r -1} dt du\\
&=T(\lam).
\end{align*}
Consequently, we get $F(\lam) =T(\lam)$ for any $\lam >0$ which then implies all inequalities in the previous estimates to be equalities. Thus, for almost $u\in S^{n-1}$ we have $\rho(u) \geq \lam^{\frac2{1+ \al -\frac\be p}}$ for any $\lam >0$. Hence, for almost $u\in S^{n-1}$, we have $\rho(u) = \infty$. Moreover, for such a $u\in S^{n-1}$ we have $J(u,t) =1$ for almost $t>0$ (which ensures the equality in the first inequality in the estimates above). By the continuity, we have $J(u,t) =1$ for any $t >0$. Using again the polar coordinate \eqref{eq:polar}, we get, for any $r>0$, $\Vol_g(B(P,r)) = \om_n r^n$. By the equality condition in Bishop--Gromov volume comparison principle theorem (see \cite[Theorem $III.4.4$]{Chavel}), $M$ is isometric to $\R^n$.
\end{proof}

\section{Acknowledgments}
The author would like to thank Professor Alexandru Krist\'aly for drawing our attentions to his works \cite{K1,K2,KO}. This work was supported by the CIMI's postdoctoral research fellowship.

\end{document}